\documentclass[11pt]{article}
\usepackage[english]{babel}
\usepackage{times}
\usepackage{paralist}
\usepackage{amsmath,amstext,amssymb,amsthm, amsfonts}
\usepackage[dvips]{graphics,color}
\numberwithin{equation}{section}

\newtheorem{lemma}{Lemma}[section]
\newtheorem{theorem}[lemma]{Theorem}
\newtheorem{corollary}[lemma]{Corollary}
\newtheorem{definition}[lemma]{Definition}
\newtheorem{proposition}[lemma]{Proposition}
\newtheorem{example}[lemma]{Example}

\DeclareMathOperator{\graph}{Gr}

\def\X{\mathbb{X}}
\def\U{\mathbb{U}}
\def\K{\mathbb{K}}
\def\N{\mathbb{N}}
\def\Y{\mathbb{Y}}
\def\S{\mathbb{S}}
\def\R{\mathbb{R}}

\title{Berge's Maximum Theorem for Noncompact Image Sets}

\begin{document}
\date{}
\maketitle

\begin{center}
 Eugene~A.~Feinberg\footnote{{This research was partially supported by NSF grants CMMI-0928490 and CMMI-1335296}\\ Department of Applied Mathematics and
Statistics,
 State University of New York at Stony Brook,
Stony Brook, NY 11794-3600, USA, eugene.feinberg@sunysb.edu}, \
Pavlo~O.~Kasyanov\footnote{Institute for Applied System Analysis,
National Technical University of Ukraine ``Kyiv Polytechnic
Institute'', Peremogy ave., 37, build, 35, 03056, Kyiv, Ukraine,\
kasyanov@i.ua.},
 and Mark Voorneveld\footnote{Department of Economics, Stockholm School of Economics, Box 6501, 113 83 Stockholm, Sweden, mark.voorneveld@hhs.se} \\

\smallskip
\end{center}

\begin{abstract}
This note generalizes Berge's maximum theorem to noncompact image
sets.  It is also clarifies the results from E.A. Feinberg, P.O.
Kasyanov, N.V. Zadoianchuk, ``Berge's theorem for noncompact image
sets," J. Math. Anal. Appl. 397(1)(2013), pp. 255--259 on the
extension to noncompact image sets of another Berge's theorem,
that states semi-continuity of value functions. Here we explain
that the notion of a $\K$-inf-compact function introduced there is
applicable to metrizable topological spaces and to more general
compactly generated topological spaces. For Hausdorff topological
spaces we introduce the notion of a $\K\N$-inf-compact function
($\N$ stands for ``nets" in $\K$-inf-compactness), which coincides
with $\K$-inf-compactness for compactly generated and, in
particular, for metrizable topological spaces.

\end{abstract}

\section{Introduction} Let $\X$ and $\Y$ be
Hausdorff topological spaces, $u:\X\times \Y\to
\overline{\mathbb{R}}=\mathbb{R}\cup\{\pm\infty\}$ and
$\Phi:\X\to\S(\Y)$, where $\S(\Y):=2^{\Y}\setminus \{\emptyset\}$
is \textit{the family of all nonempty subsets} of the set $\Y.$
Consider a value function
\begin{equation}\label{eq:lm10}
v(x):=\inf\limits_{y\in {\Phi}(x)}u(x,y),\qquad 
x\in\X.
\end{equation}

  A set-valued mapping ${F}:\X
\to \S(\Y)$ is \textit{upper semi-continuous} at $x\in\X$ if, for
any neighborhood $\mathcal{G}$ of the set $F(x)$, there is a
neighborhood of $x$, say $\mathcal{O}(x)$, such that
$F(y)\subseteq \mathcal{G}$ for all $y\in \mathcal{O}(x)$; a
set-valued mapping ${F}:\X \to \S(\Y)$ is \textit{lower
semi-continuous} at $x\in\X$ if,  for any open set $\mathcal{G}$
with $F(x) \cap \mathcal{G} \neq \emptyset$, there is a
neighborhood of $x$, say $\mathcal{O}(x)$, such that if $y\in
\mathcal{O}(x)$, then
$F(y)\cap \mathcal{G}\ne\emptyset$. 
A set-valued mapping is called \textit{upper (lower)
semi-continuous}, if it is upper (lower) semi-continuous at all
$x\in\X$. A set-valued mapping is called \textit{continuous}, if
it is upper and lower semi-continuous. For a topological space
$\U$, we denote by $\K(\U)$ the family of all nonempty compact
subsets of     \ $\U.$

For Hausdorff topological spaces, Berge's well-known maximum
theorem
(cf. Berge \cite[p.~116]{Berge}) has the following formulation.\\
\vskip -0.5 cm \noindent{\bf Berge's Maximum Theorem.} {\rm (Hu
and Papageorgiou~\cite[p.~84]{Hu}) \textit{If
$u:\X\times\Y\to{\mathbb{R}}$ is a continuous function and
$\Phi:\X\to \K(\Y)$ is a 
continuous set-valued
mapping, then the value function $v:\X\to\R$ is continuous and
the solution multifunction ${\Phi}^*:\X\to \S(\Y)$, defined as
\begin{equation}\label{eq:lm000} {\Phi}^*(x)=\left\{y\in
{\Phi}(x):\,v(x)=u(x,y)\right\},\quad x\in\X,
\end{equation}
is  upper semi-continuous and compact-valued.}

This paper extends Berge's theorem to possibly noncompact sets
$\Phi(x),$ $x\in\X.$  
For a numerical function $f$, defined on a nonempty subset $U$ of a topological space $\mathbb{U},$ consider the level sets
\begin{equation}\label{def-D}
\mathcal{D}_f(\lambda;U)=\{y\in U \, : \,  f(y)\le \lambda\},\qquad \lambda\in\R.\end{equation}  We recall that a function $f$ is \textit{lower
semi-continuous on $U$} if all the level sets $\mathcal{D}_f(\lambda;U)$
 are closed, and a function $f$ is
\textit{inf-compact} (also sometimes called \emph{lower semi-compact}) on $U$ if all these sets are compact.

For $Z \subseteq \X$, let
\[
{\rm Gr}_Z(\Phi)=\{(x,y)\in Z\times\Y\,:\, y\in\Phi(x)\}
\]
%
\begin{definition}\label{def:kinf} {\rm (Feinberg et al. \cite[Definition 1.1]{Feinberg et
al})} A function $u:\X\times \Y\to \overline{\mathbb{R}}$ is called
$\K$-inf-compact on ${\rm Gr}_{\X}(\Phi)$, 
if for every $K\in \K(\X)$ this function is inf-compact on ${\rm
Gr}_K(\Phi)$.
\end{definition}


In many applications the space $\X$ is compactly generated. Recall
that a topological space $\X$ is \emph{compactly generated\/}
(Munkres \cite[p. 283]{Munkres} or a $k$-space, Kelley
\cite[p.~230]{Kelley}, Engelking \cite[p.~152]{Engelking}) if it
satisfies the following property: each set $A \subseteq \X$ is
closed in $\X$ if $A \cap K$ is closed in $K$ for each $K \in
\K(\X)$. In particular, all locally compact spaces (hence,
manifolds) and all sequential spaces (hence, first-countable,
including metrizable/metric spaces) are compactly generated; see
Munkres \cite[Lemma 46.3, p. 283]{Munkres}, Engelking
\cite[Theorem 3.3.20, p. 152]{Engelking}.

The following theorem and its generalization for Hausdorff topological  spaces, Theorem~\ref{thm: Berge1}, 
are the main results of this
paper.

\begin{theorem}\label{thm: Berge}
Assume that:
\begin{enumerate}[(a)]
\item $\X$ is a compactly generated topological space;
\item $\Phi: \X \to \S(\Y)$ is lower semi-continuous; \item $u: \X
\times \Y \to \R$ is $\K$-inf-compact and upper semi-continuous on
$\graph_{\X}(\Phi)$.
\end{enumerate}
Then the value function $v: \X\to \R$ is continuous and the solution multifunction $\Phi^*:\X\to \K(\Y)$ is 
upper semi-continuous and compact-valued.
\end{theorem}

When $\X$ is a  compactly generated topological space,
Theorem~\ref{thm: Berge} generalizes  Berge's Maximum Theorem
because, if $\Phi$ is a compact-valued and upper semi-continuous
mapping, then a lower semi-continuous function $u: \X \times \Y
\to \R$ is $\K$-inf-compact on $\graph_{\X}(\Phi)$; Feinberg et
al.~\cite[Lemma 2.1(i)]{Feinberg et al} or its generalization,
Lemma~\ref{lm0}(i).
 Note that a more particular result than
Theorem~\ref{thm: Berge} is formulated in Feinberg et al.
\cite[Theorem~4.1]{Feinberg et al}  for Hausdorff topological
spaces, where upper semi-continuity of the solution multifunction
$\Phi^*:\X\to \K(\Y)$  is stated  for a continuous set-valued
mapping $\Phi:\X\to
\S(\Y)$ and for a continuous function $u$. 
However, when the topological space $\X$ is Hausdorff, it is
necessary to consider a more restrictive assumption for
$u(\cdot,\cdot)$ than $\K$-inf-compactness, because 
$\K$-inf-compactness of $u$ on $\graph(\Phi)$ is not sufficient
for lower semi-continuity of the value function $v$ on $\X$; see
Example~\ref{exa:Mark}. This assumption (we call it
\emph{$\K\N$-inf-compactness}) is a generalization of the
$\K$-inf-compactness property  in the way it is
formulated in Feinberg et al. \cite{Feinberg et al MDP} as
Assumption ${\rm\bf (W^*)}$(ii) for metric spaces.
\begin{definition}\label{def:F-ic}
A function $u:\X\times \Y\to \overline{\mathbb{R}}$ is called
$\K\N$-inf-compact on ${\rm Gr}_{\X}(\Phi)$, if the following two
conditions hold:
\begin{enumerate}[(i)]
\item $u(\cdot,\cdot)$ is lower semi-continuous on ${\rm
Gr}_{\X}(\Phi)$;
\item for any convergent net $\{x_i
\}_{i\in I}$ with values in $\X$ whose limit $x$ belongs to $\X$, any net
$\{y_i\}_{i\in I}$ defined on the same ordered set $I$ with $y_i\in
\Phi(x_i)$, $i\in I,$ and
satisfying the condition that the set $\{u(x_i,y_i): i\in I \}$ is bounded
above, has an accumulation point $y\in \Phi(x).$
\end{enumerate}
\end{definition}

As proved below, $\K$-inf-compactness and $\K\N$-inf-compactness
are equivalent if $\X$ is  a
compactly generated topological space (Corollary~\ref{cor2.2}).
All the statements in Feinberg et al. \cite{Feinberg et al} are
formulated for $\K$-inf-compact functions $u$.  However, the
statements of Feinberg et al. \cite[Theorems 1.2, 3.1(a), 4.1 and
Lemma 2.3]{Feinberg et al} require slightly stronger assumptions
if we want them to hold for a Hausdorff topological space $\X$.
Indeed, the proofs in 
\cite[Theorems 1.2, 3.1(a), 4.1 and Lemma 2.3]{Feinberg et al}
rely on the claim that, for any convergent net $\{x_\alpha\}$ in a
Hausdorff topological space with a limit $x$, the set
$(\cup_\alpha\{x_\alpha\})\cup\{x\}$ is a compact set. This is
true for converging sequences, but not necessarily for converging
nets (cf. Example~\ref{exa:Mark}). Restricting attention to
compactly generated spaces, which are more general objects than
metric spaces, makes all the results in \cite{Feinberg et al}
valid and makes it possible to formulate Berge's maximum theorem
for $K$-inf-compact functions and noncompact image sets; see
Theorem~\ref{thm: Berge}. For general Hausdorff topological
spaces, the $\K\N$-inf-compactness assumption is needed; see
Theorems~\ref{thm: Berge1}, \ref{zMT1}, \ref{zlsc}, and
Proposition~\ref{prop3.6}.




When $\X$ is a Hausdorff topological space, the following theorem
  is analogous to Theorem~\ref{thm: Berge}.  In view of
  Lemma~\ref{lm0}(i), Theorem~\ref{thm: Berge1} generalizes
  Berge's Maximum Theorem to possibly noncompact  sets $\Phi(x),$ $x\in\X$.
\begin{theorem}\label{thm: Berge1}
If a function $u: \X \times \Y \to \mathbb{R}$ is
$\K\N$-inf-compact and upper semi-continuous on
$\graph_{\X}(\Phi)$ and $\Phi: \X \to \S(\Y)$ is a lower
semi-continuous set-valued mapping, then the value function $v:
\X\to \R$ is continuous and the solution multifunction
$\Phi^*:\X\to \K(\Y)$ is
upper semi-continuous and compact-valued.
\end{theorem}

The papers that extend  Berge's theorems to different directions
include Ausubel and Deneckere \cite{AD}, Horsley et al.
\cite{HWVZ}, Leininger \cite{Leininger}, Montes-de-Oca and
Lemus-Rodr\'{\i}guez \cite{MdOLR}, and Walker \cite{Walker}.
Relations with provided above theorems are mostly superficial;
most papers impose stronger topological restrictions \cite{HWVZ},
assume compact-valuedness of the feasibility multifunction $\Phi$
\cite{Leininger, Walker}, or impose restrictions on the value
function $v$ rather than on the primitives of the model \cite{AD}.
None of these papers appeals to compactly generated spaces and/or
$\K/\K\N$-inf-compactness. Only the recent paper by Montes-de-Oca
and Lemus-Rodr\'{\i}guez \cite{MdOLR}  contains results relevant
to this paper. They restrict attention to metric spaces and use
inf-compactness, rather than generalizations provided above, in
addition to other restrictions to derive special cases of
presented results on the value function $v$ and solution
multifunction $\Phi^*$. In particular, the main results of
\cite{MdOLR}, Theorems 3.1 and 4.1, are corollaries  of
Theorems~\ref{thm: Berge} and \ref{thm: Berge1} above, as well as
of \cite[Theorem 4.1]{Feinberg et al} applied to metric spaces.


\section{Classification of Inf-Compactness Properties}

In this section we study the relation between
$\K\N$-inf-compactness and $\K$-inf-compactness.
%

\begin{theorem}\label{kf}
The following statements hold:
\begin{enumerate}[(i)]
\item if $u:\X\times \Y\to \overline{\mathbb{R}}$ is
$\K\N$-inf-compact on ${\rm Gr}_{\X}(\Phi)$, then the function
$u(\,\cdot\,,\,\cdot\,)$ is $\K$-inf-compact on ${\rm
Gr}_{\X}(\Phi)$; \item if $u:\X\times \Y\to \overline{\mathbb{R}}$
is $\K$-inf-compact on ${\rm Gr}_{\X}(\Phi)$ and $\X$ is a compactly
generated topological space, then the function
$u(\,\cdot\,,\,\cdot\,)$ is $\K\N$-inf-compact on ${\rm
Gr}_{\X}(\Phi)$.
\end{enumerate}
\end{theorem}

Theorem~\ref{kf} implies the following statement.
\begin{corollary}\label{cor2.2} If  $\X$ is a
compactly generated topological space, then a function $u:\X\times
\Y\to \overline{\mathbb{R}}$ is $\K\N$-inf-compact on ${\rm
Gr}_{\X}(\Phi)$ if and only if it  is $\K$-inf-compact on ${\rm
Gr}_{\X}(\Phi)$.
\end{corollary}

Before the proof of Theorem~\ref{kf}, we describe auxiliary
properties of set-valued mappings.

\subsection{Properties of  Set-Valued Mappings}

This subsection introduces $\K$-upper semi-compact and
$\K\N$-upper semi-compact set-valued mappings and relates these
two objects to each other and to upper semi-continuous set-valued
mappings.


\begin{definition}
A set-valued mapping ${\Psi}:\X \to \S(\Y)$ is $\K$-\textit{upper
semi-compact}, if for every $K\in \K(\X)$ the set ${\rm Gr}_K(\Psi)$
is compact.
\end{definition}
\begin{definition}\label{sc}
The mapping $\Psi:\X\to\S(\Y)$  is \textit{ $\K\N$-upper
semi-compact} if for any convergent net $\{x_i \}_{i\in I}$ with
values in $\X$, whose limit $x$ belongs to $\X$, any net
$\{y_i\}_{i\in I}$, defined on the same ordered set $I$ with $y_i\in
\Psi(x_i)$, $i\in I,$ has an accumulation point $y\in \Psi(x).$
\end{definition}

\begin{theorem}\label{uscusc}
The following statements hold:
\begin{enumerate}[(i)]
\item  a set-valued mapping $\Psi:\X\to \S(\Y)$ is $\K\N$-upper
semi-compact if and only if it is upper semi-continuous and
compact-valued; \item a $\K\N$-upper semi-compact set-valued
mapping $\Psi:\X\to \S(\Y)$ is $\K$-upper semi-compact; \item if
$\X$ is a compactly generated topological space and $\Psi:\X\to
\S(\Y)$ is a $\K$-upper semi-compact set-valued mapping, then
$\Psi$ is compact-valued and upper semi-continuous.
\end{enumerate}
\end{theorem}
\begin{proof}
(i) Let $\Psi:\X\to \S(\Y)$ be a $\K\N$-upper semi-compact
set-valued mapping. Restricting attention to constant nets
$\{x_i\}_{i \in I}$, it follows that $\Psi$ takes compact values.
Let us prove that $\Psi$ is upper semi-continuous. Suppose, to the
contrary, that $\Psi$ is not upper semi-continuous at some point
$x_0 \in \X$. Then there is an open neighborhood $V$ of
$\Psi(x_0)$ such that for every neighborhood $U$ of $x_0$, there
is an $x_U \in U$ with $\Psi(x_U) \not\subseteq V$. In particular,
we can select a $y_U \in \Psi(x_U) \setminus V$. Now consider the
nets $\{x_U: U \in I\}$ and $\{y_U: U \in I\}$, where $I$ is the
directed set of neighborhoods of $x_0$. The net $\{x_U\}$
converges to $x_0$. Since $\Psi$ is $\K\N$-upper semi-compact, the
net $\{y_U\}$ has an accumulation point $y_0 \in \Psi(x_0)\subseteq V$.
The net $\{y_U\}$ lies in the
closed set $V^c$, which is the complement of $V$, and therefore $y_0\in V^c$.
This contradiction implies that
$\Psi:\X\to \S(\Y)$ is upper semi-continuous.

Vice versa, let $\Psi:\X\to \K(\Y)$ be upper semi-continuous, let
$\{x_i \}_{i\in I}$ be a convergent net with values in $\X$ whose
limit $x$ belongs to $\X$ and $\{y_i\}_{i\in I}$ be a net defined on
the same ordered set $I$ with $y_i\in \Psi(x_i)$, $i\in I$.
Aliprantis and Border \cite[Corollary 17.17, p. 564]{AliBor} yields
that the net $\{y_i\}_{i\in I}$ has an accumulation point $y\in
\Phi(x)$. Therefore, the function $u(\cdot,\cdot)$ is
$\K\N$-inf-compact on ${\rm Gr}_\X ({{\Phi}})$.

(ii) Let $\Psi:\X\to \S(\Y)$ be a $\K\N$-upper semi-compact
set-valued mapping. This mapping is $\K$-upper semi-compact, because
its restriction to any compact set $K$ of $\X$ is $\K\N$-upper
semi-compact and its graph $\graph_K(\Psi)$ is compact by virtue of
characterizations of compactness via nets.

(iii) Since $\Psi:\X\to \S(\Y)$ is $\K$-upper semi-compact, it is
compact-valued. We prove that $\Psi:\X\to \S(\Y)$ is upper
semi-continuous. Recall (Aliprantis and Border \cite[Lemma 17.4,
p. 559]{AliBor}) that $\Psi$ is upper semi-continuous if, for each
closed subset $F$ of $\Y$, the set
\begin{equation}\label{eq: minimizers usc}
\{x \in \X: \Psi(x) \cap F \neq \emptyset\}
\end{equation}
is closed. Since $\X$ is compactly generated, it suffices to show
for each compact $K \subseteq \X$ that $\Psi|_K$, the restriction
of $\Psi: \X \to \S(\Y)$ to the domain $K$, is upper
semi-continuous: for each closed subset $F$ of $\Y$,  
\[
\{x \in K: \Psi|_K(x) \cap F \neq \emptyset\} = \{x \in \X: \Psi(x)
\cap F \neq \emptyset\} \cap K
\]
is closed and consequently, that the set in \eqref{eq: minimizers
usc} is closed.

So let $K \in \K(\X)$. Since $\Psi|_K$ is compact-valued, its
upper semi-continuity follows from compactness of
$\graph_K(\Psi)$, that is, for every net $(x_{\alpha},
y_{\alpha})$ in $\graph_K(\Psi)$ with $x_{\alpha} \to x$ for some
$x \in K$, that net $(y_{\alpha})$ has a limit point in
$\Psi|_K(x)$, that is, a convergent subnet with limit $y \in
\Psi|_K(x)$; see Aliprantis and Border \cite[Corollary 17.17, p.
564]{AliBor}.
\end{proof}

Theorem~\ref{uscusc} implies the following statement.

\begin{corollary}\label{cor:usc}
Let $\X$ be a compactly generated topological space. A set-valued
mapping $\Psi:\X\to \S(\Y)$ is $\K$-upper semi-compact if and only
if it is $\K\N$-upper semi-compact.
\end{corollary}

\subsection{Proof of Theorem~\ref{kf}}

\begin{proof}[Proof of Theorem~\ref{kf}]
(i) Let $u:\X\times \Y\to \overline{\mathbb{R}}$ be
$\K\N$-inf-compact on ${\rm Gr}_{\X}(\Phi)$, $K\in\K(\X)$, and
$\lambda\in \mathbb{R}$. Prove that the level set
$\mathcal{D}_{u(\cdot,\cdot)}(\lambda;{\rm Gr}_K(\Phi))$ is
compact, that is, any net $\{(x_i,y_i)\}_{i\in I}$ with values in
$\mathcal{D}_{u(\cdot,\cdot)}(\lambda;{\rm Gr}_K(\Phi))$ has an
accumulation point $(x,y)\in
\mathcal{D}_{u(\cdot,\cdot)}(\lambda;{\rm Gr}_K(\Phi))$. Indeed,
condition (ii) of Definition~\ref{def:F-ic} and compactness of $K$
implies that a net $\{(x_i,y_i)\}_{i\in I}\subset
\mathcal{D}_{u(\cdot,\cdot)}(\lambda;{\rm Gr}_K(\Phi))$ has an
accumulation point $(x,y)\in {\rm Gr}_{\X}(\Phi)$. Condition (i)
of Definition~\ref{def:F-ic} yields that the set
$\mathcal{D}_{u(\cdot,\cdot)}(\lambda;{\rm Gr}_K(\Phi))$ is
closed, that is, $(x,y)\in
\mathcal{D}_{u(\cdot,\cdot)}(\lambda;{\rm Gr}_K(\Phi))$.
Therefore, the function $u(\,\cdot\,,\,\cdot\,)$ is
$\K$-inf-compact on ${\rm Gr}_{\X}(\Phi)$.

(ii) Let $\X$ be a compactly generated topological space and
$u:\X\times \Y\to \overline{\mathbb{R}}$ be $\K$-inf-compact on
${\rm Gr}_{\X}(\Phi)$. Fix an arbitrary $\lambda\in\mathbb{R}$.
According to the definition of $\K\N$-inf-compactness, it is
sufficient to prove that: (a) the set
$\mathcal{D}_{u(\cdot,\cdot)}(\lambda;{\rm Gr}_\X(\Phi))$ is
closed, and (b) for any convergent net $\{x_i \}_{i\in I}$ with
values in $\X$ whose limit $x$ belongs to $\X$, any net
$\{y_i\}_{i\in I},$ defined on the same ordered set $I$ with
$y_i\in \Phi(x_i)$, $i\in I,$ and satisfying the condition that
the set $\{u(x_i,y_i): i\in I \}$ is bounded above by $\lambda$,
has an accumulation point $y\in \Phi(x).$

Set $\bar{\Y}:=\Y\cup\{a\}$, where $a$ is a subset of $\Y$ such
that $a\notin \Y$ (such set exists according to Cantor's theorem).
A subset $\mathcal{O}\subseteq \bar{\Y}$ is called open in
$\bar{\Y}$, if $\mathcal{O}\setminus\{a\}$ is open in $\Y$. Note
that the point $a$ is isolated and a set $\mathcal{K}\subseteq
\bar{\Y}$ is open (closed, compact) in $\bar{\Y}$ if and only if
the set $\mathcal{K}\setminus\{a\}$ is open (closed, compact
respectively) in $\Y$. Therefore, the topological space
$\bar{\Y}$, endowed with such topology of open subsets
$\mathcal{O}$, is  Hausdorff.

According to Theorem~\ref{uscusc}, the set-valued mapping
$\Phi_\lambda:\X\to \S(\bar{\Y})$,
\[
\Phi_\lambda(x):=\{y\in\Phi(x)\,:\, u(x,y)\le
\lambda\}\cup\{a\},\quad x\in \X,
\]
is compact-valued, upper semi-continuous, $\K$-upper semi-compact,
and $\K\N$-upper semi-compact, because the topological space $\X$ is
compactly generated and for every $K\in \K(\X)$ the set
\[
{\rm Gr}_K(\Phi_\lambda)=\mathcal{D}_{u(\cdot,\cdot)}(\lambda;{\rm
Gr}_K(\Phi))\cup(K\times\{a\})
\]
is compact in $\X\times \bar{\Y}$.

Since $\Phi_\lambda:\X\to \K(\bar{\Y})$ is upper semi-continuous,
the set ${\rm Gr}_\X(\Phi_\lambda)$ is closed in $\X\times\bar{\Y}$.
Therefore, the set $\mathcal{D}_{u(\cdot,\cdot)}(\lambda;{\rm
Gr}_\X(\Phi))={\rm Gr}_\X(\Phi_\lambda)\setminus (\X\times\{a\})$ is
closed in $\X\times \bar{\Y}$ and in $\X\times\Y$, because the set
$\X\times\{a\}$ is open and closed simultaneously in $\X\times
\bar{\Y}$ and ${\rm Gr}_\X(\Phi_\lambda)\cap
(\X\times\{a\})=\emptyset$. Statement (a) is proved.

Statement (b) follows from $\K\N$-upper semi-compactness of
$\Phi_\lambda:\X\to \K(\bar{\Y})$. Indeed, if $\{x_i \}_{i\in I}$
is a convergent net with values in $\X$ whose limit $x$ belongs to
$\X$, and $\{y_i\}_{i\in I}$ is a net defined on the same ordered
set $I$ with $y_i\in \Phi(x_i)$, $i\in I,$ and satisfying the
condition that the set $\{u(x_i,y_i): i\in I \}$ is bounded above
by $\lambda$, then this net has an accumulation point $y\in
\Phi(x)$, because $y_i\neq a$ for any $i\in I$ and $a$ is an
isolated point in $\bar{\Y}$. Statement (b) is proved. Since
statements (a) and (b) hold for any real $\lambda$, the function
$u$ is $\K\N$-inf-compact on ${\rm Gr}_{\X}(\Phi)$.
\end{proof}

\section{Properties of $\K$-Inf-Compact and $\K\N$-Inf-Compact Functions}\label{F}

The following theorems state some properties of $\K$-inf-compact
functions.
\begin{theorem}\label{thm: some properties of value function and minimizers}
If a function $u:\X\times \Y\to \overline{\mathbb{R}}$ is
$\K$-inf-compact on $\graph_{\X}(\Phi)$ then:
\begin{enumerate}[(a)]
\item\label{nonempty-valued} For each $x \in \X$, the set
$\Phi^*(x)$ is nonempty. \item\label{compact-valued} If $v(x) = +
\infty$, then $\Phi^*(x) = \Phi(x)$. If $v(x) < + \infty$, then
$\Phi^*(x)$ is compact. \item\label{lsc on compact sets} For each
$K \in \K(\X)$, the restriction $v|_K: K \to \overline{\R}$ is
lower semi-continuous. \item\label{lsc if compactly generated} If
$\X$ is compactly generated, the function $v: \X \to
\overline{\R}$ is lower semi-continuous.
\end{enumerate}
\end{theorem}
\begin{proof}
(\ref{nonempty-valued}), (\ref{compact-valued}). Let $x \in \X$.
If $v(x) = \inf_{y \in \Phi(x)} u(x,y) = + \infty$, then $u(x,y) =
+ \infty$ for all $y \in \Phi(x)$. Hence $\Phi^*(x) = \Phi(x)$ is
nonempty. Next, let $v(x) \in \R \cup \{- \infty\}$ and let
$\lambda \in (v(x), + \infty)$. Then $\lambda$ belongs to $\R$ and
$\mathcal{D}_{u(\cdot, \cdot)}(\lambda; \graph_{\{x\}}(\Phi))$ is
nonempty  and compact.  The former follows from  $v(x) < \lambda$,
and the latter follows from $\K$-inf-compactbess of $u$ applied to
the compact set $K=\{x\}.$  So, for $\lambda\in (v(x), + \infty)$,
the level sets $\mathcal{D}_{u(\cdot, \cdot)}(\lambda;
\graph_{\{x\}}(\Phi))$  are nonempty, compact, and have the finite
intersection property. Hence, their intersection
\[
\cap_{\lambda \in (v(x), + \infty)} \mathcal{D}_{u(\cdot, \cdot)} (\lambda, \graph_{\{x\}}(\Phi)) = \{(x,y) \in \graph_{\{x\}} (\Phi): u(x,y) \leq v(x)\} = \{x\} \times \Phi^*(x)
\]
is nonempty and compact. A fortiori, the projection $\Phi^*(x)$ onto $\Y$ is nonempty and compact.

 (\ref{lsc on compact sets}). Let $K \in \K(\X)$ and $\lambda \in \R$. To show that
$\mathcal{D}_{v(\cdot)}(\lambda; K) = \{x \in K: v(x) \leq
\lambda\}$ is closed, consider a convergent net $x_{\alpha} \to x$
in $K$ with $v(x_{\alpha}) \leq \lambda$ for all $\alpha$. By
(\ref{nonempty-valued}), there exists, for each $\alpha$, some
$y_{\alpha} \in \Phi(x_{\alpha})$ with $u(x_{\alpha}, y_{\alpha})
= v(x_{\alpha}) \leq \lambda$. So the net $(x_{\alpha},
y_{\alpha})$ belongs to $\mathcal{D}_{u(\cdot, \cdot)} (\lambda;
\graph_K (\Phi))$, which is compact by $\K$-inf-compactness.
Consequently, it has a convergent subnet (without loss of
generality, the original net) with the limit $(x,y) \in
\mathcal{D}_{u(\cdot, \cdot)} (\lambda; \graph_K (\Phi))$. In
particular, $y \in \Phi(x)$ and $v(x) \leq u(x,y) \leq \lambda$,
as we had to show.

 (\ref{lsc if compactly generated}). Let $\X$ be
compactly generated and let $\lambda \in \R$. By (\ref{lsc on
compact sets}), $ \mathcal{D}_{v(\cdot)}(\lambda; K) =
\mathcal{D}_{v(\cdot)}(\lambda; \X) \cap K $ is closed for each $K
\in \K(\X)$. Therefore, $\mathcal{D}_{v(\cdot)}(\lambda; \X)$ is
closed. \end{proof}
\begin{corollary}\label{zcor1} {\rm(cf.  Feinberg and Lewis
\cite[Proposition~3.1]{FL}, Feinberg et al.
\cite[Corollary~3.2]{Feinberg et al})} If a function
$u:\X\times\Y\to\overline{\mathbb{R}}$ is inf-compact on ${\rm
Gr}_{\X}(\Phi)$, then the function $v:\X\to\overline{\R}$ is
inf-compact and the conclusions of Theorem \ref{zMT1} hold.
\end{corollary}
\begin{proof}  In view of Theorem~\ref{thm: some properties of value function and
minimizers}, $\Phi^*(x)\ne \emptyset$ and the function $v(x)$ is
defined for all $x\in\X$. For any $\lambda\in\R$, the level set
$\mathcal{D}_{v(\cdot)}(\lambda;\X)$ is compact as the projection
of the compact set $\mathcal{D}_{u(\cdot,\cdot)}(\lambda;{\rm
Gr}_\X(\Phi))$ on $\X$. Thus the function $v(\cdot)$ is
inf-compact.
\end{proof}

For an upper semi-continuous set-valued mapping ${\Phi}:\X\to
\K(\Y)$, the set ${\rm Gr}_\X(\Phi)$ is closed;
Berge~\cite[Theorem 6, p. 112]{Berge}. Therefore, for such $\Phi$,
if a function $u(\cdot,\cdot)$ is lower semi-continuous on
$\X\times\Y$, then it is lower semi-continuous on ${\rm
Gr}_\X(\Phi)$. Thus, Lemma~\ref{lm0}(i) implies that
Theorems~\ref{thm: Berge} and \ref{thm: Berge1} are natural
generalizations of Berge's Maximum Theorem.  
Lemma~\ref{lm0} generalizes  \cite[Lemma~2.1]{Feinberg et al}.

\begin{lemma}\label{lm0}
The following statements hold:

(i) if $u:\X\times\Y\to \overline{\mathbb{R}}$ is lower
semi-continuous on ${\rm Gr}_\X ({{\Phi}})$ and ${\Phi}:\X\to
\K(\Y)$ is upper semi-continuous, then the function
$u(\cdot,\cdot)$ is $\K\N$-inf-compact on ${\rm Gr}_\X
({{\Phi}})$;

(ii) if $u:\X\times\Y\to \overline{\mathbb{R}}$ is inf-compact on ${\rm Gr}_\X(\Phi)$, then the function $u(\cdot,\cdot)$ is $\K\N$-inf-compact on ${\rm
Gr}_\X ({{\Phi}})$ and, therefore, it is $\K$-inf-compact on ${\rm Gr}_\X(\Phi)$.
\end{lemma}
\begin{proof}
(i) Let $\{x_i \}_{i\in I}$ be a convergent net with values in
$\X$ whose limit $x$ belongs to $\X$ and $\{y_i\}_{i\in I}$ be a
net defined on the same ordered set $I$ with $y_i\in \Phi(x_i)$,
$i\in I,$ and satisfying the condition that the set $\{u(x_i,y_i):
i\in I \}$ is bounded above by $\lambda\in\mathbb{R}$. Let us
prove that a net $\{y_i\}_{i\in I}$ has an accumulation point
$y\in \Phi(x)$ such that $u(x,y)\le \lambda$. Aliprantis and
Border \cite[Corollary 17.17, p. 564]{AliBor} yields that a net
$\{y_i\}_{i\in I}$ has an accumulation point $y\in \Phi(x)$. The
lower semi-continuity of $u$ on $\graph(\Phi)$ implies that
$u(x,y)\le \lambda$. Therefore, the function $u(\cdot,\cdot)$ is
$\K\N$-inf-compact on ${\rm Gr}_\X ({{\Phi}})$.

(ii) The function $u$ is lower semi-continuous on ${\rm Gr}_\X
({{\Phi}})$ because the level set $\mathcal{D}_{u(\cdot, \cdot)}
(\lambda, \graph_{\X}(\Phi))$ is compact and, therefore, it is
closed for any $\lambda\in \mathbb{R}$. Let us prove that for any
convergent net $\{x_i \}_{i\in I}$ with values in $\X$ whose limit
$x$ belongs to $\X$, any net $\{y_i\}_{i\in I}$ defined on the
same ordered set $I$ with $y_i\in \Phi(x_i)$, $i\in I,$ and
satisfying the condition that the set $\{u(x_i,y_i): i\in I \}$ is
bounded above, has an accumulation point $y\in \Phi(x).$ This
holds because the level set $\mathcal{D}_{u(\cdot, \cdot)}
(\lambda, \graph_{\X}(\Phi))$ is compact for any
$\lambda\in\mathbb{R}$ and because of the characterizations of
compactness via nets.
\end{proof}
As explained above, Theorems~1.2,  4.1 and Lemma~2.3 from Feinberg
et al. \cite{Feinberg et al}  are proved  there, in fact, for
$\K\N$-inf-compact functions $u$. In addition, the proofs of
\cite[Theorem~3.1(a) and Corollary~3.2] {Feinberg et al} use
\cite[Theorem~1.2 and Lemma~2.3] {Feinberg et al}. If the space
$\X$ is compactly generated (in particular, a metrizable
topological space is compactly generated), a function $u$ is
$\K\N$-inf-compact
if and only if it is $\K$-inf-compact; Corollary~\ref{cor2.2}. 
Below we state the corrected formulations of Feinberg et al.
\cite[Theorem3~1.2, 3.1(a) and
 Lemma~2.3]{Feinberg et al} for Hausdorff topological spaces $\X$ and $\Y$.  We do not provide the
corrected formulation of Feinberg et al. \cite[Theorem~4.1]{Feinberg et al} because Theorem~\ref{thm: Berge1} is a stronger statement.

\begin{theorem}\label{zMT1} {\rm(cf. Feinberg et
al. \cite[Theorem~1.2]{Feinberg et al})} If the function $u:\X\times
\Y\to \overline{\mathbb{R}}$ is $\K\N$-inf-compact on ${\rm
Gr}_{\X}(\Phi)$, then the function $v:\X\to\overline{\mathbb{R}}$ is
lower semi-continuous.
\end{theorem}

\begin{lemma}\label{zlsc} {\rm(cf. Feinberg et
al. \cite[Lemma~2.3]{Feinberg et al})} A $\K\N$-inf-compact function
$u(\cdot,\cdot)$ on ${\rm Gr}_{\X}(\Phi)$ is lower semi-continuous
on ${\rm Gr}_{\X}(\Phi)$.
\end{lemma}

\begin{proposition}\label{prop3.6} {\rm(cf. Feinberg et
al. \cite[Theorem 3.1(a)]{Feinberg et al})} If the function
$u:\X\times \Y\to \overline{\mathbb{R}}$ is $\K\N$-inf-compact on
${\rm Gr}_{\X}(\Phi)$, then ${\rm Gr}_{\X}(\Phi^*)$ is a Borel
subset of $\X\times\Y.$
\end{proposition}

%
%

%
%



In addition,  \cite[Corollary 3.2]{Feinberg et al} is restated
above as Corollary~\ref{zcor1} with the   proof that does not use
  \cite[Theorem 1.2]{Feinberg et al}. Example~\ref{exa:Mark} demonstrate that the
  conclusions of Theorem~\ref{zMT1} and Lemma~\ref{zlsc} do not
  hold if $u$ is $\K$-inf-compact on ${\rm Gr}_{\X}(\Phi)$ and
  $\X$ is Hausdorff.

\section{Proofs of Theorems~\ref{thm: Berge} and \ref{thm: Berge1}}

%
%


According to Corollary~\ref{cor2.2}, Theorem~\ref{thm: Berge} is a
direct corollary of Theorem~\ref{thm: Berge1}. 

\begin{proof}[Proof of Theorem~\ref{thm: Berge1}]
The function $u$ is continuous  on $\graph_\X(\Phi)$, because it
is upper semi-continuous  on $\graph_\X(\Phi)$ and, according to
Definition~\ref{def:F-ic}, it is lower semi-continuous on
$\graph_\X(\Phi)$. In view of Theorems~\ref{thm: some properties
of value function and minimizers} and \ref{zMT1}, the value
function $v: \X\to \R$ is continuous and the solution set-valued
mapping $\Phi^*:\X\to \K(\Y)$ is
compact-valued. 
%
Let us show that the solution multi-function $\Phi^*$ is
$\K\N$-upper semi-compact.  Consider a net $\{x_i \}_{i\in I}$
with values in $\X$ whose limit $x$ belongs to $\X$. Then any net
$\{y_i\}_{i\in I},$ defined on the same ordered set $I$ with
$y_i\in \Phi^*(x_i)$, $i\in I,$ has an accumulation point $y\in
\Phi^*(x).$  Indeed, $v(x_i)=u(x_i,y_i)$ for any $i\in I$. Since
$x_i\to x$ and $v$ is continuous, a net $\{v(x_i)\}$ is bounded
above by a finite constant eventually in $I$. Therefore,
$\K\N$-inf-compactness of the function $u$ on $\graph_\X(\Phi)$
implies that the net $\{y_i\}_{i\in I}$ has an accumulation point
$y\in \Phi(x).$ Since the functions $u$ and $v$ are continuous on
$\graph_\X(\Phi)$ and
$\X$ respectively, 
$y\in\Phi^*(x)$. Thus, the solution multifunction $\Phi^*$ is
$\K\N$-upper semi-compact and, in view of Theorem~\ref{uscusc}(i),
it is upper semi-continuous.
\end{proof}

\section{Counterexample}

In the example below, $\X$ is a Hausdorff topological space, $\Y$
is a singleton, $\Phi:\X\to \Y$ is a continuous mapping, $u$ is a
$\K$-inf-compact real-valued function  on $\graph_\X(\Phi)$ such
that: (i) $u$ is not lower semi-continuous on $\X\times \Y$ and
(ii) the value function $v$ is not lower semi-continuous on $\X$
either.


\begin{example}\label{exa:Mark}
{\rm Consider a space $[0, \omega_1]$ of ordinals in the order
topology, with $\omega_1$ the least uncountable ordinal. Each
non-limit ordinal $\alpha$ is an isolated point: $\alpha = 0$ is
isolated, since $\{0\} = [0,1)$ is open. And if $\alpha \neq 0$,
there is a $\beta$ with $\beta + 1 = \alpha$. Hence $\{\alpha\} =
(\beta, \alpha + 1)$ is open. Now let $\X$ be the subspace
consisting of all non-limit ordinals and $\omega_1$.

A set in $\X$ is compact if and only if it is finite. To see why an infinite set $X \subseteq \X$ is not compact, let $C \subseteq X \setminus \{\omega_1\}$ be a countably infinite subset. Identifying each $c \in C$ with its countable set of predecessors $\{x \in [0, \omega_1]: x < c\}$ and using that a countable union of countable sets is countable, it follows that the supremum (union) $s$ of $C$ in $[0, \omega_1]$ is countable and consequently satisfies $s < \omega_1$. Then $(s, \omega_1]$ is an open set containing $\omega_1$. Since $(s, \omega_1]$ fails to cover the infinitely many isolated terms $c \in C$, collection $\{\{x\}: x \in X \setminus \{\omega_1\}\} \cup \{(s, \omega_1]\}$ is a cover of $X$ without a finite subcover.

Now let $\Y = \{y\}$ be a singleton set, $\Phi(x) = \{y\}$ for each $x \in \X$, and define $u: \X \times \Y \to \R$ by $u(x,y) = 0$ if $x \neq \omega_1$ and $u(\omega_1, y) = 1$. Since $\Y$ and each compact subset of $\X$ are finite, $u$ is $\K$-inf-compact. But $u$ is not lower semi-continuous: the set $A$ of non-limit ordinals is directed in its usual order and net $(x_{\alpha}, y)_{\alpha \in A}$ with $x_{\alpha} = \alpha$ converges to $(\omega_1, y)$, yet
\[
\liminf_{\alpha} u(x_{\alpha}, y) = 0 < 1 = u(\omega_1,y),
\]
contradicting lower semi-continuity.

This also shows that the value function $v:x \mapsto \inf_{y \in \Phi(x)} u(x,y)$ need not be lower semi-continuous: in this example, $\Phi(x) = \{y\}$, so $v(x) = u(x,y)$ for all $x \in \X$ so that we have an analogous violation of lower semi-continuity:
\[
\liminf_{\alpha} v(x_{\alpha}) = 0 < 1 = v(\omega_1).
\]
}
\end{example}

\end{document}